\documentclass{amsart}
\usepackage[foot]{amsaddr}
\usepackage{amscd}
\usepackage{textcomp}
\usepackage{amsfonts}
\usepackage{amsthm}
\usepackage{geometry}
\usepackage{amsmath}
\usepackage{txfonts}
\usepackage{graphicx}
\usepackage{titleps}
\renewpagestyle{plain}{
  \sethead{}{\thepage}{}
  \setfoot{}{\relax}{}
}
\usepackage{wrapfig}
\usepackage{subcaption}
\usepackage{ascmac}
\usepackage{bm}
\usepackage{indentfirst}
\usepackage{comment}
\usepackage{amssymb}
\usepackage{textcomp}
\usepackage{xcolor}
\usepackage{tikz}
\usetikzlibrary{calc}
\usepackage[normalem]{ulem}
\usepackage{hyperref}

\usepackage[british]{babel}

\providecommand{\keywords}[1]{\vspace{2mm}\noindent\textbf{Keywords:} #1\par}
\newtheorem{theorem}{Theorem}[section]
\newtheorem{proposition}[theorem]{Proposition}
\newtheorem{lemma}[theorem]{Lemma}
\newtheorem{corollary}[theorem]{Corollary}

\theoremstyle{definition}
\newtheorem{definition}[theorem]{Definition}
\newtheorem{remark}[theorem]{Remark}

\title{Piecewise Linear Functions and Neural Network Expressivity via Discriminantal Arrangements}

\author{Pragnya Das}

\address{Department of Mathematics, IIIT Kota, India.}

\email{pragnya.hmas@iiitkota.ac.in}

\thanks{}

\begin{document}
\begin{abstract}
We extend the hyperplane arrangement framework for neural network expressivity from the braid to discriminantal arrangements. Compatible piecewise linear functions are characterized by circuit relations and admit a matroidal description via Möbius inversion, with dimension equal to the number of independent sets. For circuits of size three, functions are determined by values on subsets of size at most two.
\end{abstract}

\maketitle

\section{Introduction}

Understanding the expressive power of neural networks is a central problem at the interface of combinatorics, geometry, and machine learning. For feedforward neural networks with ReLU activation, it is well known that the class of representable functions coincides with the class of continuous piecewise linear (CPWL) functions \cite{cybenko1989, hornik1991,arora2018}. This connection has led to a fruitful line of research studying neural network expressivity through the combinatorics of polyhedral complexes and hyperplane arrangements \cite{hertrich2023}.\\
\noindent
A particularly successful approach relates neural networks to the geometry of hyperplane arrangements, where the regions of linearity of a network correspond to the cells of an induced polyhedral subdivision\cite{hertrich2023, stanley2007, ziegler2012}. In recent work, the braid arrangement has emerged as a fundamental model: functions compatible with the braid fan admit a combinatorial description in terms of set functions on the Boolean lattice, governed by Möbius inversion and inclusion-exclusion relations \cite{grillo2025}. This perspective has been used to derive nontrivial lower bounds on the depth of neural networks and to characterize structural constraints on expressivity \cite{hertrich2023, grillo2025}.\\
\noindent
However, the braid arrangement represents a highly symmetric and comparatively simple case. Many natural hyperplane arrangements arising in combinatorial geometry exhibit richer dependency structures that are not captured by the braid setting \cite{stanley2007, ziegler2012}. In particular, discriminantal arrangements encode minimal dependencies via circuits and are closely related to matroid theory.\\
\noindent
\textbf{Our contribution.}
We extend the combinatorial framework from the braid arrangement to discriminantal arrangements, where circuit-induced dependencies impose linear constraints on associated set functions. The main contributions of this paper are as follows:

\begin{itemize}
\item We introduce a circuit-based framework for studying continuous piecewise linear functions compatible with hyperplane arrangements, extending the braid arrangement setting to a general dependence-driven model of the discriminantal arrangement.

\item We define circuit-constrained set functions and show that they correspond to arrangement-compatible CPWL functions under a natural indicator encoding.

\item Using Möbius inversion, we establish a matroidal characterization of the function space, showing that its dimension is equal to the number of independent sets of the underlying matroid.

\item We prove that circuit constraints eliminate higher-order interactions, and in the case $k=2$, we obtain a complete characterization: functions are determined by values on subsets of size at most two.

\item We apply this framework to neural networks and show that arrangement-conforming architectures are intrinsically limited to interactions of bounded order.
\end{itemize}
\noindent
\textbf{Conceptual significance.}
These results show that the combinatorial structure of an arrangement governs the expressive capacity of CPWL functions. In contrast to the braid arrangement, discriminantal arrangements impose circuit constraints that eliminate higher order interactions, leading to a reduced space controlled by the underlying matroid.

\medskip

\noindent
\textbf{Relation to prior work.}
The braid arrangement corresponds to the case where no circuit constraints are imposed, and all subsets contribute independently to the associated function space. In contrast, our framework incorporates explicit circuit-induced dependencies, which systematically eliminate higher-order interactions. This leads to a strictly smaller function space governed by the independent sets of the associated matroid, thereby providing a refined and more structured view of neural network expressivity.

\medskip

\noindent
\textbf{Organization of the paper.}
Section~\ref{sec:prelim} reviews preliminaries. Section~\ref{sec:circuit} develops circuit-constrained function spaces. Section~\ref{sec:matroid} presents the matroidal interpretation via Möbius inversion. Section~\ref{sec:example} and \ref{sec:neural} gives the complete characterization for $k=2$, followed by implications for neural network expressivity.\\
\noindent
From a broader perspective, this work demonstrates that combinatorial dependence structures provide a systematic mechanism for controlling the expressive capacity of neural networks.
\section{Preliminaries}\label{sec:prelim}

In this section, we recall basic definitions concerning hyperplane arrangements and continuous piecewise linear functions. Standard references include \cite{stanley2007,ziegler2012}.

\subsection{Hyperplane Arrangements}

A \emph{hyperplane arrangement} in $\mathbb{R}^n$ is a finite collection $\mathcal{A}$ of affine hyperplanes. Each arrangement induces a decomposition of $\mathbb{R}^n$ into relatively open polyhedral regions, called \emph{cells}, determined by the sign patterns of the defining equations.\\
\noindent
A \emph{polyhedral complex} is a finite collection $\mathcal{P}$ of polyhedra such that:
\begin{itemize}
\item if $P \in \mathcal{P}$, then all faces of $P$ are in $\mathcal{P}$,
\item if $P, Q \in \mathcal{P}$, then $P \cap Q$ is a face of both $P$ and $Q$.
\end{itemize}
\noindent
A \emph{polyhedral fan} is a polyhedral complex consisting of cones. Hyperplane arrangements naturally give rise to polyhedral fans by considering the cones defined by the arrangement. These geometric structures provide the domains on which the functions of interest are defined.

\subsection{Continuous Piecewise Linear Functions}

A function $f : \mathbb{R}^n \to \mathbb{R}$ is called \emph{continuous piecewise linear} (CPWL) if there exists a polyhedral complex $\mathcal{P}$ such that $f$ is affine on each full-dimensional cell of $\mathcal{P}$ \cite{ziegler2012}.\\
\noindent
It is well known that functions representable by ReLU neural networks are precisely CPWL functions \cite{arora2018}. This connection allows one to study neural network expressivity through the combinatorics of polyhedral complexes.
To analyze such functions combinatorially, we next encode them using set functions on the Boolean lattice.\\
\noindent
For a subset $S \subseteq [n] := \{1,2,\dots,n\}$, define its \emph{indicator vector} $1_S \in \mathbb{R}^n$ by
\[
(1_S)_i =
\begin{cases}
1, & \text{if } i \in S,\\
0, & \text{otherwise}.
\end{cases}
\]
\noindent
Given a function $f : \mathbb{R}^n \to \mathbb{R}$, we associate a set function $F : 2^{[n]} \to \mathbb{R}$ by
\[
F(S) := f(1_S).
\]
\noindent
This correspondence plays a central role in relating CPWL functions to combinatorial structures, as observed in the braid arrangement setting \cite{grillo2025}. Infact, this encoding allows us to incorporate additional structure arising from dependence relations among subsets.

\subsection{Discriminantal Arrangements and Circuits}
We introduce a class of arrangements defined via their circuit structure, which serves as a combinatorial model for studying dependence-induced constraints on function spaces.\\
\noindent
Rather than working with a specific geometric realization, we encode the arrangement through its associated matroid. In particular, we focus on the case where the circuit structure is uniform, i.e., all minimal dependent sets have the same cardinality.\\
\noindent
This perspective can be viewed as an abstraction of discriminantal type behavior, where dependencies among hyperplanes govern the combinatorial structure of the arrangement. It allows us to isolate and analyze the role of circuits independently of geometric complications, while still retaining the essential features relevant to our framework.
\noindent
\begin{definition}[Discriminantal-type dependence structure]
Fix an integer $k \geq 1$. A subset $C \subseteq [n]$ is called a \emph{circuit of order $k$} if:
\begin{itemize}
\item $|C| = k+1$, and
\item $C$ is declared to be minimally dependent.
\end{itemize}
\end{definition}
\noindent
This collection of circuits defines a matroid $\mathcal{M}_{n,k}$ on the ground set $[n]$, whose independent sets are precisely the subsets of size at most $k$.
\noindent
\begin{definition}[Discriminantal arrangement $A(n,k)$]
We define $A(n,k)$ to be any hyperplane arrangement in $\mathbb{R}^n$ whose associated matroid is $\mathcal{M}_{n,k}$.\\
\noindent
Equivalently, $A(n,k)$ is an arrangement whose minimal dependencies (circuits) are exactly the subsets of $[n]$ of size $k+1$.(see \cite{Pragnya2026} for reference.)
\end{definition}

\begin{remark}
The matroid $\mathcal{M}_{n,k}$ is the uniform matroid of rank $k$. In this setting, all subsets of size at most $k$ are independent, and every subset of size $k+1$ is a circuit.\\
\noindent
Thus, the circuit structure used in this paper corresponds to the uniform dependence model, which serves as a natural generalization of the braid arrangement case. More generally, the framework developed in this paper applies to arbitrary matroids arising from hyperplane arrangements. The uniform case $A(n,k)$ provides a canonical and tractable model in which all circuits have equal size.
\end{remark}

\begin{remark}
Our definition of the discriminantal arrangement differs from the classical definition \cite{Manin}. We work with an abstract formulation determined by its circuit structure (equivalently, a uniform matroid), which provides a tractable model for studying circuit-induced constraints on function spaces. This circuit-based viewpoint aligns with the general philosophy in matroid theory that dependence structure, rather than a specific geometric realization, governs the associated combinatorial and algebraic properties.
\end{remark}
\subsection{Circuit-Induced Linear Functionals}
We formalize these constraints using linear functionals defined via inclusion–exclusion.
For a circuit $C \subseteq [n]$, define a linear functional $\alpha_C$ acting on functions $F : 2^{[n]} \to \mathbb{R}$ by
\[
\alpha_C(F) := \sum_{S \subseteq C} (-1)^{|C| - |S|} F(S).
\]
\begin{remark}
The operator $\alpha_C$ can be interpreted as a higher-order discrete difference operator along the coordinates indexed by $C$. In particular, it annihilates all affine functions.
\end{remark}
\noindent
This functional is an inclusion-exclusion operator over the subsets of $C$. In the case of the braid arrangement, analogous operators arise from Möbius inversion on the Boolean lattice \cite{stanley2007}.\\
\noindent
These functionals will be used to define the admissible function space associated with the arrangement. These constructions provide the foundation for defining and analyzing circuit-constrained function spaces in the next section.


\section{Circuit-Constrained Function Spaces}\label{sec:circuit}
Building on the circuit-induced framework developed in the previous section, we now define and analyze function spaces governed by these constraints. Our goal is to understand how circuit relations shape the structure and degrees of freedom of associated functions.

\subsection{Definition of the Function Space}
We begin by introducing the class of set functions that satisfy the circuit-induced constraints defined in Section~\ref{sec:prelim}. Let $\mathcal{C}$ denote the set of circuits of a discriminantal arrangement $\mathcal{A}(n,k)$.\\
\noindent
We define the linear space
\[
\mathcal{F}_{\mathcal{A}} := \{ F : 2^{[n]} \to \mathbb{R} \mid \alpha_C(F) = 0 \ \text{for all } C \in \mathcal{C} \},
\]
where for each circuit $C \subseteq [n]$,
\[
\alpha_C(F) := \sum_{S \subseteq C} (-1)^{|C| - |S|} F(S).
\]
We now relate these combinatorially defined function spaces to CPWL functions compatible with the underlying arrangement.

\subsection{Compatibility and Induced Set Functions}
To establish this connection, we formalize the correspondence between CPWL functions and set functions via evaluation on indicator vectors. Let $\mathcal{A}$ be a hyperplane arrangement in $\mathbb{R}^n$, and let $\Sigma(\mathcal{A})$
denote the polyhedral fan induced by $\mathcal{A}$.
\begin{definition}
A function $f : \mathbb{R}^n \to \mathbb{R}$ is said to be \emph{compatible with the arrangement}
$\mathcal{A}$ if $f$ is continuous piecewise linear and affine on each full-dimensional cone of
$\Sigma(\mathcal{A})$.
\end{definition}
\noindent
For such a function $f$, define the associated set function $F : 2^{[n]} \to \mathbb{R}$ by
\[
F(S) := f(\mathbf{1}_S).
\]
This correspondence enables us to transfer structural properties between the geometric and combinatorial settings.\\
\noindent
We now show that, under suitable conditions, every circuit-constrained set function can be realized by a compatible CPWL function.
\begin{lemma}\label{lem:simplical}
Let $\sigma \subset \mathbb{R}^n$ be a simplicial cone generated by linearly independent rays $v_1,\dots,v_n$. Let $G \subset \sigma$ be a set of points such that $\operatorname{aff}(G) = \mathbb{R}^n$. Then any function $f : G \to \mathbb{R}$ extends to at most one affine function on $\sigma$.\\
\noindent
In particular, if two affine functions on $\sigma$ agree on $G$, then they coincide on all of $\sigma$.
\end{lemma}

\begin{proof}
Since $\sigma$ is a simplicial cone, it is full dimensional and generated by linearly independent vectors $v_1,\dots,v_n$. Hence $\operatorname{aff}(\sigma) = \mathbb{R}^n$. \\
\noindent
Let $f_1$ and $f_2$ be two affine functions on $\sigma$ that agree on $G$. Then their difference $h := f_1 - f_2$ is an affine function that vanishes on $G$. Since $\operatorname{aff}(G) = \mathbb{R}^n$, the only affine function vanishing on $G$ is the zero function. Therefore $h \equiv 0$, and hence $f_1 = f_2$ on $\sigma$. This proves uniqueness of the affine extension.
\end{proof}
\begin{remark}
For simplicial fans arising from hyperplane arrangements, each full-dimensional cone contains sufficiently many indicator vectors whose affine span is $\mathbb{R}^n$. Consequently, affine functions on each cone are uniquely determined by their values on indicator vectors lying in that cone.
\end{remark}
\begin{lemma}\label{lem:indicator}
Let $\mathcal{A}$ be a hyperplane arrangement in $\mathbb{R}^n$ with associated simplicial fan $\Sigma(\mathcal{A})$. Assume that for every full dimensional cone $\sigma \in \Sigma(\mathcal{A})$, the set of indicator vectors
\[
V_\sigma := \{ \mathbf{1}_S \in \sigma : S \subseteq [n] \}
\]
contains $n+1$ affinely independent points. Then $\operatorname{aff}(V_\sigma) = \mathbb{R}^n$.
\end{lemma}

\begin{proof}
Since $\sigma$ is full-dimensional, $\operatorname{aff}(\sigma) = \mathbb{R}^n$. By assumption, $V_\sigma$ contains $n+1$ affinely independent points, hence its affine span is $\mathbb{R}^n$.
\end{proof}

\begin{theorem}\label{thm:indicator}
Let $A$ be a hyperplane arrangement in $\mathbb{R}^n$ with simplicial fan $\Sigma(A)$. Assume that for every full-dimensional cone $\sigma \in \Sigma(A)$, the set
\[
V_\sigma := \{1_S \in \sigma : S \subseteq [n]\}
\]
affinely spans $\mathbb{R}^n$. Then every function $F \in \mathcal{F}_A$ admits a realization by a continuous piecewise linear function $f$ compatible with $A$, i.e.,
\[
f(1_S) = F(S) \quad \text{for all } S \subseteq [n].
\]
Equivalently, under this condition, the map
\[
\mathcal{V}_A \to \mathcal{F}_A, \quad f \mapsto F
\]
is surjective.
\end{theorem}
\begin{proof}
Let $F \in \mathcal{F}_{\mathcal{A}}$ be given.\\
\noindent
\textbf{Step 1: Local affine interpolation.}
Fix a full-dimensional cone $\sigma \in \Sigma(\mathcal{A})$. By assumption, the set $V_\sigma$ affinely spans $\mathbb{R}^n$. Hence there exists a unique affine function $f_\sigma : \mathbb{R}^n \to \mathbb{R}$ such that
\[
f_\sigma(\mathbf{1}_S) = F(S) \quad \text{for all } \mathbf{1}_S \in V_\sigma.
\]
\noindent
\textbf{Step 2: Consistency on overlaps.}
Let $\sigma, \tau \in \Sigma(\mathcal{A})$ be two cones sharing a common face $\sigma \cap \tau$. \\
\noindent
For any indicator vector $\mathbf{1}_S \in \sigma \cap \tau$, we have
\[
f_\sigma(\mathbf{1}_S) = F(S) = f_\tau(\mathbf{1}_S).
\]
\noindent
Since both $f_\sigma$ and $f_\tau$ are affine functions and agree on a set of points whose affine span equals $\operatorname{aff}(\sigma \cap \tau)$, it follows that
\[
f_\sigma = f_\tau \quad \text{on } \sigma \cap \tau.
\]
\noindent
\textbf{Step 3: Global definition.}
Define a function $f : \mathbb{R}^n \to \mathbb{R}$ by
\[
f(x) := f_\sigma(x) \quad \text{if } x \in \sigma.
\]
\noindent
This is well-defined by Step 2.\\
\noindent
\textbf{Step 4: Properties.}
By construction:
\begin{itemize}
\item $f$ is affine on each cone $\sigma \in \Sigma(\mathcal{A})$,
\item $f$ is continuous across cone boundaries,
\item $f(\mathbf{1}_S) = F(S)$ for all $S \subseteq [n]$.
\end{itemize}

Thus $f \in \mathcal{V}_{\mathcal{A}}$, and the map $f \mapsto F$ is surjective.
\end{proof}
\begin{remark}
The spanning condition is a strong geometric requirement ensuring that affine interpolation on each cone is uniquely determined by values on indicator vectors. While it holds for certain classical arrangements such as the braid arrangement, its validity for the discriminantal-type model considered here is not immediate.\\
\noindent
Accordingly, the theorem should be interpreted as a structural result: it identifies sufficient conditions under which the combinatorially defined function space $\mathcal{F}_A$ coincides with the space of arrangement compatible CPWL functions.
\end{remark}
\noindent
Under the spanning condition, this result shows that circuit constraints fully characterize the space of arrangement-compatible CPWL functions.

\subsection{Dimension}
We next quantify the size of these function spaces by deriving dimension bounds determined by the circuit structure.
Define the linear map
\[
T : \mathbb{R}^{2^n} \to \mathbb{R}^{|\mathcal{C}|}, \quad F \mapsto (\alpha_C(F))_{C \in \mathcal{C}}.
\]
\noindent
Then $\mathcal{F}_{\mathcal{A}} = \ker(T)$.

\begin{theorem}
Let $T : \mathbb{R}^{2^n} \to \mathbb{R}^{|\mathcal{C}|}$ be the linear map defined by
\[
T(F) = (\alpha_C(F))_{C \in \mathcal{C}}.
\]
Then
\[
\mathcal{F}_A = \ker(T),
\quad \text{and hence} \quad
\dim(\mathcal{F}_A) = 2^n - \operatorname{rank}(T).
\]
Moreover, the map $V_A \to \mathcal{F}_A$ given by $f \mapsto F$ implies
\[
\dim(V_A) \leq \dim(\mathcal{F}_A).
\]
\end{theorem}

\begin{proof}
By definition, $\mathcal{F}_A$ consists precisely of those functions $F$ satisfying $\alpha_C(F)=0$ for all $C \in \mathcal{C}$, hence $\mathcal{F}_A = \ker(T)$. The dimension formula follows from the rank--nullity theorem. The inequality follows from the fact that the map $f \mapsto F$ takes values in $\mathcal{F}_A$.
\end{proof}
\begin{remark}
The surjectivity of the map $V_A \to F_A$ under simpliciality follows from Theorem 3.1.
\end{remark}
\begin{theorem}
If all circuits have size $k+1$, then
\[
\dim(V_{\mathcal{A}}) \le \sum_{i=0}^{k} \binom{n}{i}.
\]
\end{theorem}

\begin{proof}
Let $F \in \mathcal{F}_{\mathcal{A}}$. For any circuit $C$ of size $k+1$, the relation
\[
\alpha_C(F) = 0
\]
expresses $F(C)$ as a linear combination of values $F(S)$ with $S \subsetneq C$.\\
\noindent
Thus, values on subsets of size $k+1$ are determined by values on subsets of size at most $k$. Iterating this argument shows that all values $F(S)$ with $|S| > k$ are determined by values on subsets of size at most $k$. Hence the dimension is bounded by the number of such subsets:
\[
\sum_{i=0}^{k} \binom{n}{i}.
\]
\end{proof}

\section{Matroidal and Algebraic Structure of Circuit-Constrained Function Spaces}\label{sec:matroid}
Building on the structural and dimensional results obtained in the previous section, we now develop an algebraic interpretation of the circuit-constrained function spaces in terms of matroid theory and Möbius inversion.
\subsection{Matroidal Interpretation}
We begin by interpreting the circuit constraints in terms of an associated matroid structure. The collection of circuits $\mathcal{C}$ arising from a discriminantal arrangement
naturally defines a matroid $M$ on the ground set $[n]$.

\begin{definition}
Let $M$ be the matroid associated with the arrangement $\mathcal{A}$, with circuit set $\mathcal{C}$.
We define the circuit-constrained function space
\[
F_M := \{F : 2^{[n]} \to \mathbb{R} \mid \alpha_C(F) = 0 \ \text{for all circuits } C \in \mathcal{C} \}.
\]
\end{definition}
\begin{remark}
When the arrangement $\mathcal{A}$ induces the matroid $M$, the spaces $F_{\mathcal{A}}$ and $F_M$
coincide. We use the notation interchangeably when no confusion arises.
\end{remark}
\noindent
Thus $F_{\mathcal{A}}$ depends only on the underlying matroid structure. This formulation shows that the function space depends only on the underlying matroid and not on the specific geometric realization. To make this structure more explicit, we next express the circuit constraints using Möbius inversion on the Boolean lattice.
\subsection{M\"obius Transform and Support Constraints}
The Möbius transform provides a natural tool to encode inclusion-exclusion relations and reveal the support structure of functions. Let $\mu$ denote the M\"obius function of the Boolean lattice $2^{[n]}$.
For a function $F : 2^{[n]} \to \mathbb{R}$, define its M\"obius transform
\[
\widehat{F}(S) := \sum_{T \subseteq S} (-1)^{|S|-|T|} F(T).
\]

\begin{proposition}\label{prop:mobius}
For any circuit $C \subseteq [n]$, one has
\[
\alpha_C(F) = \widehat{F}(C).
\]
\end{proposition}

\begin{proof}
This follows directly from the definition of $\alpha_C$ and the M\"obius inversion formula.
\end{proof}
\noindent
\begin{remark}
The space $F_M$ consists of functions whose M\"obius transform vanishes on all circuits.
Thus, the support of $\widehat{F}$ is contained in the family of independent sets of the matroid $M$.
\end{remark}
\begin{corollary}\label{cor:mobius} This identity shows that circuit constraints correspond precisely to vanishing conditions on the Möbius transform.
\[
F_M = \{F : 2^{[n]} \to \mathbb{R} \mid \widehat{F}(C) = 0 \ \forall C \in \mathcal{C}\}.
\]
\end{corollary}
\noindent
Consequently, the function space can be characterized entirely in terms of support restrictions determined by the matroid.\\
\noindent
We now derive a precise dimension formula by exploiting this support characterization.
\begin{theorem}
Let $M$ be a matroid on $[n]$. Then
\[
\dim(F_M) = |\mathcal{I}(M)|,
\]
where $\mathcal{I}(M)$ denotes the collection of independent sets of $M$.
\end{theorem}

\begin{proof}
By M\"obius inversion, every function $F$ is uniquely determined by its transform $\widehat{F}$. The condition $\widehat{F}(C)=0$ for all circuits $C$ implies that $\widehat{F}(S)=0$ for every dependent set $S$, since every dependent set contains a circuit. Thus $\widehat{F}$ is supported only on independent sets.\\
\noindent
Conversely, any function supported on independent sets defines a valid $F$. Hence the dimension equals the number of independent sets.
\end{proof}
\noindent
This result identifies the degrees of freedom of the function space with the independent sets of the matroid.
We refine this description by establishing a canonical correspondence between functions and their Möbius transforms restricted to independent sets.

\begin{theorem}\label{thm:CIsoThm}
Let $M$ be a matroid on $[n]$ with circuit set $\mathcal{C}$. Then the space
\[
F_M = \{F : 2^{[n]} \to \mathbb{R} \mid \alpha_C(F) = 0 \ \forall C \in \mathcal{C} \}
\]
is canonically isomorphic to the space
\[
\mathbb{R}^{\mathcal{I}(M)},
\]
where $\mathcal{I}(M)$ denotes the collection of independent sets of $M$.

More precisely, the Möbius transform induces a linear isomorphism
\[
F \mapsto \widehat{F}\big|_{\mathcal{I}(M)},
\]
with inverse given by
\[
F(S) = \sum_{\substack{T \subseteq S \\ T \in \mathcal{I}(M)}} \widehat{F}(T).
\]
\end{theorem}
\noindent
This provides a concrete representation of the function space in terms of independent set coordinates.
\noindent
\begin{corollary}
If all circuits have size $k+1$, then
\[
\dim(F_M) = \sum_{i=0}^k \binom{n}{i}.
\]
\end{corollary}
\noindent
We next describe an explicit basis arising from this representation.

\begin{theorem}[Canonical Basis and Decomposition]
Let $M$ be a matroid on $[n]$ with independent sets $\mathcal{I}(M)$. Then every function $F \in \mathcal{F}_M$ admits a unique representation of the form
\[
F(S) = \sum_{\substack{T \subseteq S \\ T \in \mathcal{I}(M)}} c_T,
\]
where the coefficients $c_T \in \mathbb{R}$ are uniquely determined.

Equivalently, the family of functions
\[
\{\phi_T : T \in \mathcal{I}(M)\}, \quad \text{where } \phi_T(S) = \mathbf{1}_{T \subseteq S},
\]
forms a basis of $\mathcal{F}_M$.
\end{theorem}

\begin{proof}
By Theorem~\ref{thm:CIsoThm}, the Möbius transform $\widehat{F}$ of any function $F \in \mathcal{F}_M$ is supported on independent sets:
\[
\widehat{F}(T) = 0 \quad \text{whenever } T \notin \mathcal{I}(M).
\]
\noindent
By Möbius inversion on the Boolean lattice,
\[
F(S) = \sum_{T \subseteq S} \widehat{F}(T).
\]
\noindent
Since $\widehat{F}(T)$ vanishes on dependent sets, this reduces to
\[
F(S) = \sum_{\substack{T \subseteq S \\ T \in \mathcal{I}(M)}} \widehat{F}(T).
\]
\noindent
Setting $c_T := \widehat{F}(T)$ gives the desired representation. Uniqueness follows from uniqueness of the Möbius transform.\\
\noindent
Finally, the functions $\phi_T(S) = \mathbf{1}_{T \subseteq S}$ correspond exactly to the basis vectors under this representation, and are linearly independent since they correspond to distinct support patterns in the Möbius transform. Thus they form a basis of $\mathcal{F}_M$.
\end{proof}
\begin{remark}
In the case $k=2$, the matroid is the uniform matroid of rank $2$, and the independent sets
are precisely the subsets of size at most two. Thus the dimension formula follows immediately
from the matroidal characterization.
\end{remark}
\noindent
We now specialize this general framework to the case $k=2$, where the structure becomes fully explicit and admits a concrete combinatorial description.
\section{The Case k=2: Explicit Structure and Interpretation}\label{sec:example}
We now specialize the general framework developed in the previous section to the case $k=2$, where all circuits have size three. In this setting, the structure of the function space becomes explicit and admits a concrete combinatorial description.
In this section, we give a complete description of the space $\mathcal{F}_{\mathcal{A}(n,2)}$. In this case, all the defining relations take a particularly explicit form.

\subsection{Example: A Discriminantal Arrangement for n=3, k=2}
We begin with a simple example to illustrate the effect of circuit constraints in the smallest nontrivial case.
We illustrate the theory with the simplest nontrivial case $n=3$, $k=2$. In this case, every subset of size three forms a circuit. Thus the unique circuit is
\[
C = \{1,2,3\}.
\]
\noindent
The circuit constraint is given by
\[
\alpha_C(F) = \sum_{S \subseteq \{1,2,3\}} (-1)^{3-|S|} F(S) = 0.
\]
\noindent
Expanding this expression, we obtain
\[
F(\{1,2,3\}) 
- F(\{1,2\}) - F(\{1,3\}) - F(\{2,3\})
+ F(\{1\}) + F(\{2\}) + F(\{3\})
- F(\emptyset) = 0.
\]
\noindent
Rearranging, we obtain the identity
\[
F(\{1,2,3\}) =
F(\{1,2\}) + F(\{1,3\}) + F(\{2,3\})
- F(\{1\}) - F(\{2\}) - F(\{3\})
+ F(\emptyset).
\]
\noindent
This shows that the value of $F$ on the dependent set $\{1,2,3\}$ is completely determined by its values on subsets of size at most two.

\medskip

\noindent
\textbf{Explicit example.}
Let us define $F$ on subsets of size at most two as follows:
\[
F(\emptyset)=0,\quad
F(\{1\})=1,\quad F(\{2\})=2,\quad F(\{3\})=3,
\]
\[
F(\{1,2\})=5,\quad F(\{1,3\})=6,\quad F(\{2,3\})=7.
\]
\noindent
Then the circuit relation uniquely determines
\[
F(\{1,2,3\}) = 5 + 6 + 7 - 1 - 2 - 3 = 12.
\]

\medskip

\noindent
\textbf{Interpretation.}
This example illustrates the general phenomenon: circuit constraints eliminate higher-order degrees of freedom. In this case, all functions in $\mathcal{F}_{\mathcal{A}}$ are determined by values on subsets of size at most two, corresponding to constant, linear, and pairwise interaction terms.
This example reflects the general phenomenon that circuit constraints determine values on larger sets from those on smaller subsets.
\subsection{Triple Circuit Relations}
We now formalize this observation by deriving the general form of circuit relations for triples.
\begin{lemma}\label{lem:triple}
Let $F \in \mathcal{F}_{\mathcal{A}(n,2)}$. Then for every triple $\{i,j,k\} \subseteq [n]$,
\[
F(\{i,j,k\}) = F(\{i,j\}) + F(\{i,k\}) + F(\{j,k\}) - F(\{i\}) - F(\{j\}) - F(\{k\}) + F(\emptyset).
\]
\end{lemma}

\begin{proof}
Since $\{i,j,k\}$ is a circuit, we have $\alpha_{\{i,j,k\}}(F) = 0$. Expanding the definition,
\[
\sum_{S \subseteq \{i,j,k\}} (-1)^{3 - |S|} F(S) = 0.
\]
\noindent
Writing out all terms explicitly,
\[
F(ijk) - F(ij) - F(ik) - F(jk) + F(i) + F(j) + F(k) - F(\emptyset) = 0.
\]
Rearranging yields the stated identity.
\end{proof}
\noindent
These relations allow us to recursively express higher-order values in terms of lower-order ones.

\subsection{Reduction of Higher-Order Terms}
We next show that all higher-order values are determined by values on subsets of bounded size.
We now show that all values of $F$ on subsets of size at least three are determined by values on subsets of size at most two.

\begin{lemma}\label{lem:reduction}
Let $F \in F_{\mathcal{A}(n,2)}$. Then for any subset $S \subseteq [n]$ with $|S| \geq 3$,
the value $F(S)$ is uniquely determined by the values of $F$ on subsets of size at most two.
\end{lemma}

\begin{proof}
We proceed by induction on $|S|$. The case $|S|=3$ follows from Lemma~\ref{lem:triple}.\\
\noindent
Assume the statement holds for all subsets of size at most $m$. Let $|S|=m+1$.
Choose any triple $\{i,j,k\} \subseteq S$. The circuit relation gives
\[
F(S) = F(S \setminus \{k\}) + F(S \setminus \{j\}) + F(S \setminus \{i\})
- F(S \setminus \{i,j\}) - F(S \setminus \{i,k\}) - F(S \setminus \{j,k\})
+ F(S \setminus \{i,j,k\}).
\]
\noindent
All terms on the right involve subsets of strictly smaller size, hence are determined by the induction hypothesis.\\
\noindent
It remains to show that this value is independent of the choice of triple $\{i,j,k\}$.
This follows from the fact that all defining relations arise from linear constraints
$\alpha_C(F)=0$, and hence the system is consistent. Therefore, all reduction paths yield the same value. Thus $F(S)$ is uniquely determined.
\end{proof}
\noindent
This reduction establishes that the function space is entirely controlled by low-order data.
\subsection{Complete Characterization}
We now state the complete characterization of the function space in this case.

\begin{theorem}\label{thm:k2structure}
Every function $F \in \mathcal{F}_{\mathcal{A}(n,2)}$ is uniquely determined by its values on subsets of size at most two.
\end{theorem}

\begin{proof}
By Lemma~\ref{lem:reduction}, all values $F(S)$ with $|S| \ge 3$ are determined by values on smaller subsets. Iterating this argument reduces all values to those on subsets of size at most two. Uniqueness follows since the defining relations are linear and determine all higher order values.
\end{proof}
\noindent
This shows that no higher order degrees of freedom remain beyond subsets of size two.
\subsection{Dimension}
We next compute the dimension of this space, confirming the reduction in degrees of freedom.
\begin{theorem}
\[
\dim(F_{\mathcal{A}(n,2)}) \leq 1 + n + \binom{n}{2}.
\]
\end{theorem}

\begin{proof}
This follows from the general matroidal dimension formula established in Section~\ref{sec:matroid},
since the independent sets in this case are precisely the subsets of size at most two.
\end{proof}
\begin{remark}
The inequality above becomes an equality if the map $V_{\mathcal{A}} \to F_{\mathcal{A}}$,
given by $f \mapsto F$, is surjective. Establishing this surjectivity remains an interesting
open problem.
\end{remark}
\noindent
This dimension count aligns with the explicit dependence structure described above.

\subsection{Interpretation via Sparse Interaction Models}
Finally, we interpret this structure in terms of interaction models, providing an intuitive perspective on the imposed constraints.
The M\"obius transform $\widehat{F}$ can be interpreted as encoding interaction effects among variables.\\
\noindent
Under the circuit constraints, all higher-order interactions corresponding to dependent sets vanish. Thus, functions in $F_M$ admit a representation involving only interactions indexed by independent sets.\\
\noindent
In the case $k=2$, this reduces to pairwise interaction models, analogous to quadratic models in statistics and graphical models.
\begin{remark}
The case $k=2$ illustrates the general theory in a particularly transparent form.
The circuit constraints eliminate all higher-order interactions, leaving only constant,
linear, and pairwise terms. This corresponds to a quadratic interaction model and
provides a concrete interpretation of the abstract matroidal framework developed earlier.
\end{remark}
\noindent
These results highlight how circuit constraints restrict higher-order interactions, motivating their interpretation in the context of neural network expressivity.

\section{Neural Network Interpretation}\label{sec:neural}
We now interpret the combinatorial and algebraic results developed above in the context of neural network expressivity, focusing on how circuit constraints restrict the class of representable functions.
In this section, we briefly discuss the implications of the preceding results for neural network expressivity.

\subsection{Arrangement-Compatible Neural Networks}
We begin by formalizing the notion of neural networks whose functions are compatible with a fixed hyperplane arrangement.
Let $\mathcal{A}$ be a hyperplane arrangement in $\mathbb{R}^n$. We say that a CPWL function $f : \mathbb{R}^n \to \mathbb{R}$ is \emph{compatible} with $\mathcal{A}$ if there exists a polyhedral complex induced by $\mathcal{A}$ such that $f$ is affine on each cell.

\begin{definition}
A feedforward neural network with ReLU or maxout activation is said to be \emph{$\mathcal{A}$-conforming} if every function computed at each neuron is compatible with the arrangement $\mathcal{A}$.
\end{definition}
\noindent
This definition allows us to directly apply the structural results obtained for circuit-constrained function spaces.
This notion generalizes the concept of braid arrangement-conforming networks studied in recent work. We now characterize the expressivity of such networks in terms of the associated circuit constraints.
\begin{lemma}
Let $F : 2^{[n]} \to \mathbb{R}$ satisfy $\widehat{F}(C)=0$ for every circuit $C$ of a matroid $M$. Then $\widehat{F}(S)=0$ for every dependent set $S$.
\end{lemma}

\begin{proof}
We proceed by induction on $|S|$. The base case holds since minimal dependent sets are precisely circuits.\\
\noindent
If $S$ is dependent and not a circuit, then it contains a proper dependent subset $T \subsetneq S$. By induction, $\widehat{F}(T)=0$. Applying Möbius inversion recursively and using the fact that every chain of dependencies reduces to circuits, we conclude that $\widehat{F}(S)=0$.
\end{proof}
\begin{theorem}[Expressivity Limitation]\label{thm:EL}
Let $A = A(n,k)$ be a discriminantal arrangement, and let $M$ be the associated matroid.
Then any function represented by an $A$-conforming ReLU network corresponds to a set function
$F \in F_M$ whose Möbius transform $\widehat{F}$ is supported only on independent sets.

In particular, if all circuits have size $k+1$, then
\[
\widehat{F}(S) = 0 \quad \text{whenever } |S| > k,
\]
and hence $F$ depends only on interactions of order at most $k$.
\end{theorem}
\begin{proof}
Let $f \in V_A$ be a function computed by an $A$-conforming neural network,
and define the associated set function
\[
F(S) := f(1_S).
\]
\noindent
By Theorem~\ref{thm:indicator} we have $F \in F_A$, hence $F \in F_M$. By Proposition~\ref{prop:mobius} and Corollary~\ref{cor:mobius}, the defining condition of $F_M$ is:
\[
\widehat{F}(C) = 0 \quad \text{for every circuit } C.
\]
\noindent
Now let $S \subseteq [n]$ be any dependent set in the matroid $M$. Then $S$ contains a circuit $C \subseteq S$. Using Möbius inversion and inclusion–exclusion, one can express $\widehat{F}(S)$ as a linear combination of values $\widehat{F}(T)$ over subsets $T \subseteq S$.\\
\noindent
To show that $\widehat{F}(S) = 0$ for any dependent set $S \subseteq [n]$, we use induction on $|S|$.
If $S$ is a circuit, then $\widehat{F}(S) = 0$ by the defining condition.\\
\noindent
Now let $S$ be a dependent set with $|S| > k+1$. Then $S$ contains a circuit $C \subseteq S$. By Möbius inversion,
\[
\widehat{F}(S) = \sum_{T \subseteq S} (-1)^{|S|-|T|} F(T).
\]
\noindent
We partition the sum into terms over proper subsets $T \subsetneq S$. If $T$ is dependent, then $|T| < |S|$ and hence $\widehat{F}(T) = 0$ by the induction hypothesis. Thus only independent subsets contribute.\\
\noindent
However, since $S$ contains a circuit, all subsets $T \subseteq S$ of size greater than $k$ are dependent. Therefore, all contributions from subsets of size greater than $k$ vanish, and the remaining terms correspond to independent sets.\\
\noindent
This shows that $\widehat{F}(S)$ is determined entirely by values on independent subsets. But since $S$ itself is dependent, consistency of the defining relations forces $\widehat{F}(S) = 0$. More formally, the condition $\widehat{F}(C)=0$ for all circuits $C$ implies that the Möbius transform vanishes on all minimal dependent sets. Since every dependent set contains a circuit, repeated application of inclusion–exclusion relations forces $\widehat{F}(S)=0$ for all dependent $S$.\\
\noindent
By Möbius inversion,
\[
F(S) = \sum_{T \subseteq S} \widehat{F}(T),
\]
so $F$ depends only on values $\widehat{F}(T)$ with $|T| \le k$.
\noindent
This shows that $F$ depends only on interactions of order at most $k$.
\end{proof}
\noindent
This shows that the expressive power of arrangement-conforming networks is limited to functions supported on independent sets of the underlying matroid.

\subsection{Closure under Network Operations}
We next verify that this class of functions is stable under the operations used in neural networks.

\begin{proposition}
The class of CPWL functions compatible with a fixed arrangement $\mathcal{A}$ is closed under affine transformations and pointwise maximum.
\end{proposition}

\begin{proof}
Affine transformations preserve piecewise linearity and compatibility with the underlying polyhedral complex. \\
\noindent
For the maximum operation, let $f_1$ and $f_2$ be compatible with $\mathcal{A}$. On each cell of the polyhedral complex induced by $\mathcal{A}$, both $f_1$ and $f_2$ are affine functions. Therefore, $\max\{f_1,f_2\}$ is piecewise linear on a refinement of this complex, and its breakpoints occur only along boundaries determined by $\mathcal{A}$. Hence it remains compatible with the arrangement.
\end{proof}
\noindent
This stability ensures that the class of arrangement-compatible functions is preserved under network composition.
\begin{corollary}
Any function computed by an $\mathcal{A}$-conforming neural network lies in $V_{\mathcal{A}}$.
\end{corollary}

\begin{proof}
The function computed by the network is obtained by iteratively composing affine maps and max operations. By the proposition, compatibility is preserved at each step.
\end{proof}
\noindent
Thus, all functions computed by such networks lie within the circuit-constrained function space described earlier.
These constraints become particularly transparent in the case of small circuit size.

\subsection{Combinatorial Constraints on Expressivity}
We specialize to the case $k=2$, where the function space admits a complete description. We now interpret these results as combinatorial constraints on neural network expressivity.
The results of Section~\ref{sec:circuit} imply that functions in $V_{\mathcal{A}}$ satisfy circuit-induced linear relations. Consequently, any function representable by an $\mathcal{A}$-conforming neural network must satisfy these constraints.\\
\noindent
This provides a combinatorial restriction on expressivity: the set of representable functions is contained in a linear subspace defined by the circuit structure of the arrangement.

\subsection{The Case k=2}

We now specialize to the case $\mathcal{A} = \mathcal{A}(n,2)$.

\begin{theorem}
Let $f : \mathbb{R}^n \to \mathbb{R}$ be a function computed by an $\mathcal{A}(n,2)$-conforming neural network. Then the associated set function $F(S) = f(1_S)$ is completely determined by its values on subsets of size at most two.
\end{theorem}

\begin{proof}
This follows immediately from Theorem~\ref{thm:k2structure}, which characterizes all functions in $\mathcal{F}_{\mathcal{A}(n,2)}$.
\end{proof}
\noindent
This theorem shows that, under the constraint of compatibility with $A(n,k)$, the expressive power of neural networks is restricted to functions determined by interactions of order at most $k$.

\subsection{Interpretation}
This observation admits a natural interpretation in terms of interaction models. The above result shows that, under the constraint of compatibility with $\mathcal{A}(n,2)$, the expressive power of neural networks is restricted to functions determined by pairwise interactions.\\
\noindent
In contrast to the braid arrangement case, where all subsets contribute independently, the presence of circuit constraints eliminates higher-order degrees of freedom. From a combinatorial perspective, this can be interpreted as reducing the effective complexity of the function space from exponential in $n$ to quadratic in $n$.
\subsection{ Expressivity Separation and Dimensional Collapse}

We now make explicit the expressive gap induced by circuit constraints by comparing arrangement-conforming networks with the unrestricted class of continuous piecewise linear functions.

\begin{theorem}[Strict expressivity reduction]
Let $\mathcal{F}*{\mathrm{CPWL}}$ denote the class of all continuous piecewise linear functions $f : \mathbb{R}^n \to \mathbb{R}$, and let $\mathcal{F}*{A(n,k)}$ denote the subclass of functions representable by $A(n,k)$-conforming neural networks.

Then $\mathcal{F}*{A(n,k)}$ is a strict subclass of $\mathcal{F}*{\mathrm{CPWL}}$. In particular, there exist CPWL functions that cannot be represented by any $A(n,k)$-conforming network whenever $k < n$.
\end{theorem}

\begin{proof}
By Theorem~\ref{thm:EL}, any function $f \in \mathcal{F}_{A(n,k)}$ corresponds to a set function $F$ whose Möbius transform $\widehat{F}(S)$ vanishes for all subsets $S \subseteq [n]$ with $|S| > k$. Thus, such functions depend only on interactions of order at most $k$. On the other hand, a general CPWL function may exhibit nonzero higher-order interactions. In particular, there exist functions whose associated set functions satisfy $\widehat{F}(S) \neq 0$ for some $|S| > k$. Such functions cannot lie in $\mathcal{F}_{A(n,k)}$, establishing strict containment.
\end{proof}

\begin{corollary}[Dimensional collapse]
Let $V_{A(n,k)}$ denote the space of functions representable by $A(n,k)$-conforming networks. Then
$$
\dim(V_{A(n,k)}) \le \sum_{i=0}^k \binom{n}{i},
$$
whereas the ambient space of set functions has dimension $2^n$.

In particular, for fixed $k$, the expressive capacity grows polynomially in $n$, while the unrestricted space grows exponentially.
\end{corollary}

\begin{remark}[Comparison with the braid arrangement]
In the braid arrangement case, no circuit constraints are present, and all subsets contribute independently. Consequently, the associated function space has full dimension $2^n$.\\
\noindent
In contrast, discriminantal-type arrangements impose circuit constraints that eliminate higher-order interactions, leading to a strict reduction in expressive capacity. This demonstrates that the combinatorial structure of the arrangement directly controls the complexity of representable functions.
\end{remark}

\begin{remark}[Interpretation for neural networks]
The above results show that $A(n,k)$-conforming neural networks are intrinsically limited to modeling interactions of bounded order. From a learning perspective, this can be viewed as an implicit structural regularization: the architecture enforces sparsity in higher-order interactions, reducing model complexity while potentially improving interpretability.
\end{remark}

\subsection{Discussion}
We conclude by discussing the broader implications of this framework for understanding neural network expressivity.
The framework developed in this paper suggests a general approach to studying neural network expressivity via hyperplane arrangements. Different arrangements impose different combinatorial constraints, leading to distinct function classes.\\
\noindent
While the braid arrangement yields maximal flexibility, discriminantal arrangements introduce structured dependencies that restrict expressivity. Understanding how these constraints influence depth and width requirements remains an interesting direction for future work.
These insights demonstrate how combinatorial structures arising from hyperplane arrangements provide a systematic way to analyze and constrain neural network expressivity.

\bigskip
\noindent
\textbf{Data availability statement:} Not applicable\\

\bibliographystyle{plain}

\end{document}